\documentclass{article}
\usepackage{latexsym}
\usepackage{leftidx}
\usepackage{slashed}
\usepackage{amsmath,amsthm, pdfpages}
\usepackage{enumerate}
\usepackage{amssymb}
\usepackage{upgreek}

\usepackage[margin=1.3in,dvips]{geometry}

\def\r{\rangle}

\def\f12{\frac 1 2}

\def\a{\alpha}
\def\b{\beta}
\def\ga{\gamma}

\def\Lb{\underline{L}}

\def\pa{\partial}
\def\les{\lesssim}

\def\R{\mathbb{R}}
\def\L{L}
\def\Lu{\underline{L}}

\newtheorem{thm}{Theorem}

\newtheorem{Prop}{Proposition}[section]

\newtheorem{Lem}{Lemma}[section]

\newtheorem{Remark}{Remark}[section]
\newtheorem{remark}{Remark}
\allowdisplaybreaks

\begin{document}

\title{Asymptotic decay for defocusing semilinear wave equations in $\mathbb{R}^{1+1}$}
\author{Dongyi Wei \and Shiwu Yang}

\AtEndDocument{ {\footnotesize%
  \addvspace{\medskipamount}
  \textsc{School of Mathematical Sciences, Peking University, Beijing, China} \par
  \textit{E-mail address}: \texttt{jnwdyi@pku.edu.cn} \par

  \addvspace{\medskipamount}
  \textsc{Beijing International Center for Mathematical Research, Peking University, Beijing, China} \par
  \textit{E-mail address}: \texttt{shiwuyang@bicmr.pku.edu.cn}
  }}

\date{}

\maketitle

\begin{abstract}
This paper is devoted to the study of  asymptotic behaviors of solutions to the one-dimensional defocusing semilinear wave equation.
We prove that finite energy solution tends to zero in the pointwise sense, hence improving the averaged decay of Lindblad and Tao \cite{tao12:1d:NLW}. Moreover, for sufficiently localized data belonging to some weighted energy space, the solution decays in time with an inverse polynomial rate. This confirms a conjecture raised in the mentioned work.

The results are based on new weighted vector fields as multipliers applied to regions enclosed by light rays. The key observation for the first result is an integrated local energy decay for the potential energy, while the second result relies on a type of weighted Gagliardo-Nirenberg inequality.



\end{abstract}

\section{Introduction}
In this paper, we study the global asymptotic behaviors for solutions of  the following defocusing semilinear wave equation
\begin{equation}
  \label{eq:NLW:semi:1d}
  \Box\phi=-\partial_t^2\phi + \partial_x^2\phi=|\phi|^{p-1}\phi,\quad \phi(0, x)=\phi_0(x),\quad \pa_t\phi(0, x)=\phi_1(x)
\end{equation}
in $\mathbb{R}^{1+1}$ with $p>1$. Due to the dispersive nature of the equation, the associated energy is coercive. Hence the local in time solution in energy space can be extended to the whole spacetime. However nothing too much can be said regarding the global dynamics  except that the solution is uniformly bounded.

In the beautiful work \cite{tao12:1d:NLW} by Lindblad and Tao, they showed that finite energy solution enjoys the average decay estimate
$$\lim\limits_{T\to+\infty}\frac{1}{T}\int_0^T\|\phi(t)\|_{L_x^{\infty}(\R)}dt=0.$$
A direct and important consequence of this result is that the solution asymptotically does not approach to any linear wave. This is in vast contrast to the situation in higher dimensions, where it has been shown that solutions to the energy subcritical defocusing semilinear wave equations behave like linear waves in certain sense for sufficiently large $p$. For such results, we refer to \cite{yang:NLW:2D}, \cite{yang:scattering:NLW}, \cite{yang:NLW:ptdecay:3D} and references therein.

The key reason underlying this is that higher dimensional free wave decays in time while one-dimensional wave does not. In particular, at least perturbative method can be used to study the asymptotic behaviors for solutions of the nonlinear wave equation \eqref{eq:NLW:semi:1d} in higher dimensions. For the one dimensional case studied in this paper, the above averaged decay estimate can only tell us that the solution does not scatter to linear wave. But it remains unknown what is the solution scatters to  even with small initial data.
Motivated by this, the aim of this paper is to give quantitative asymptotic decay properties for solutions of the nonlinear wave equations \eqref{eq:NLW:semi:1d}.

To state our main results, define the weighted energy norm of the initial data
\begin{equation}
\label{eq:ID:NW:1d}
\mathcal{E}_{\ga}[\phi]= \int_{\R}(1+|x|)^{\ga}(|\partial_x\phi_0|^2+|\phi_1|^2+\frac{2}{p+1}|\phi_0|^{p+1})dx,\quad \forall \gamma\geq 0.
\end{equation}
Our first result is to show that for finite energy solution, the potential energy decays to zero. As a consequence, the solution also decays to zero in the pointwise sense.
\begin{thm}
\label{thm:potential:decay}
Consider the Cauchy problem to the defocusing semilinear wave equation \eqref{eq:NLW:semi:1d} in $\mathbb{R}^{1+1}$ with finite energy data, that is, $\mathcal{E}_{0}[\phi]<+\infty. $ Then the solution $\phi$ is globally in time and verifies the following decay properties:
  $$\lim_{t\to\infty}\|\phi(t)\|_{L_x^{p+1}(\R)}=0,\quad\lim_{t\to\infty}\|\phi(t)\|_{L_x^{\infty}(\R)}=0.$$
  \end{thm}
Due to the energy conservation, the standard Gagliardo-Nirenberg inequality shows that the pointwise decay is a consequence of the potential energy decay.
 This is partly inspired by the related works \cite{Pecher82:NLW:2d}, \cite{Pecher82:decay:3d}, \cite{Velo87:decay:NLW} in higher dimensions, in which time
 decay of the potential energy is the first step toward the asymptotic behaviors of the solutions.

However the above improved decay estimates are still far from precisely describing the asymptotics of the solution. As conjectured by Lindblad and Tao in
\cite{tao12:1d:NLW}, the solution should decay in time at a polynomial rate instead of merely qualitative decay estimates. We now give this conjecture an
affirmative answer.
\begin{thm}
\label{thm:polynomial:decay}
Consider the Cauchy problem to the defocusing semilinear wave equation \eqref{eq:NLW:semi:1d}. For initial data $(\phi_0,\phi_1)$ bounded in $\mathcal{E}_{1}[\phi],$ the solution $\phi$ satisfies the following decay estimates:
\item
\begin{itemize}
\item In the exterior region when $|x|\geq |t|$, we have
\begin{align*}
  &|\phi(t,x)|\leq C(1+|t|+|x|)^{-\frac{1}{p+3}}(1+|x|-|t|)^{-\frac{1}{p+3}}.
\end{align*}
\item In the interior region $|x|< |t|$, the solution verifies the decay estimate
\begin{align*}
  &|\phi(t,x)|\leq C(1+|t|+|x|)^{\frac{1-2\a}{p+3}}(1+|t|-|x|)^{\frac{1-2\b}{p+3}}
\end{align*}
for all constants $\a$, $\b$ such that
$$  \left(\frac{1}{\a}-1\right)\left(\frac{1}{\b}-1\right)=\frac{4}{(p+1)^2}, \quad \frac{1}{2}\leq \a<1.$$
\end{itemize}
In particular, we have the uniform time decay of the solution
$$ \|\phi(t, x)\|_{L_x^{\infty}(\R)}\leq C(1+|t|)^{-\frac{p-1}{(p+1)^2+4}}.$$ Here $C$ depends only on $p$, $\a$, $\b$ and the weighted energy $\mathcal{E}_{1}[\phi].$
\end{thm}
\begin{remark}
The uniform time decay for the solution follows by taking $\b=\f12$. One can also choose
\begin{align*}
\a=\b=\frac{p+1}{p+3}
\end{align*}
to conclude that the solution verifies the following decay estimate
\begin{align*}
|\phi(t, x)|\leq C (1+t+|x|)^{-\frac{p-1}{(p+3)^2}}(1+|t-|x||)^{-\frac{p-1}{(p+3)^2}},\quad \forall t\geq 0.
\end{align*}
This improves the decay rate on the region far away from the light rays.
\end{remark}

\begin{remark}
Although the constant $C$ relies on the power $p$, it is uniformly bounded in the limit $p\rightarrow 1$ in the exterior region. In particular, the decay estimate in the exterior region is sharp in the sense that it is consistent with the decay of solutions to linear Klein-Gordon equation in the limiting case $p=1$.  The reason for this is that the vector field method used for studying wave equation also works for Klein-Gordon equations in the exterior region, see for example \cite{Yang:mMKG}.
\end{remark}

\begin{remark}
It is possible to improve the decay estimates in the interior region. But it seems that there is not too much room for doing this due to the fact that the solution can not decay too fast in view of Theorem \ref{thm:potential:decay}.
\end{remark}

Although there is significant difference for the energy subcritical defocusing semilinear equation in one dimension and higher dimensions, a common feature is that the solution exhibits certain decay properties. The key estimate in the work of Lindblad and Tao is the improved potential energy decay
\begin{align*}
\int_{t_0-T}^{t_0+T}\int_{x_0+vt-R}^{x_0+vt+R}|\phi(t, x)|^{p+1}dxdt\leq C(\sqrt{RT}+R^{-1}T),\quad \forall T\geq R>0
\end{align*}
on parallelogram, derived by using the vector field $v\pa_t+\pa_x$ as multiplier. The averaged decay estimate of the solution then follows by using the classical Rademacher differentiation theorem.

One difficulty to study the asymptotic decay for one dimensional semilinear wave equation is the lack of conformal symmetry. In higher dimensions,
there is a critical power $p=\frac{d+3}{d-1}$ such that the equation is conformally invariant. And the equation can be further classified into subconformal
case and superconformal case. Since higher dimensional linear wave decays in time, the problem becomes easier for larger $p$.
The importance of this conformal symmetry is that it directly leads to the time decay of the potential energy
\begin{align*}
\int_{\mathbb{R}^d}|\phi(t, x)|^{p+1}dx \leq C (1+ t)^{\max\{d+1-p(d-1), -2\}}
\end{align*}
by using the conformal Killing vector field as multiplier,
see \cite{Pecher82:NLW:2d}, \cite{Pecher82:decay:3d}, \cite{Velo87:decay:NLW} and recent improvements \cite{yang:scattering:NLW}, \cite{yang:NLW:2D}. However, this method is only effective when the power $p$ has a lower bound close to the conformal power. It fails for the one dimensional problem studied in this paper when $d=1$ due to the fact that equation \eqref{eq:NLW:semi:1d} is subconformal.

To overcome these difficulties and partly inspired by the work of Lindblad and Tao, we make use of new vector fields as multipliers to obtain weighted energy estimates. The key observation for the proof of Theorem \ref{thm:potential:decay} is the following integrated local energy decay for the potential energy
\begin{align*}
\int_0^{+\infty}\int_{-t-1}^{t+1}\frac{((t+1)^2-x^2)|\phi(t,x)|^{p+1}}{(t+1)^3}dxdt\leq C,
\end{align*}
derived by using the vector field
\[
(1+(1+t)^{-2}x^2)\pa_t+2(1+t)^{-1}x\pa_x
\]
as multiplier. This estimate plays the role that it indicates that the potential energy concentrates mainly in a neighborhood of the light ray $\{t=|x|\}$. On the region close to the light ray, we can bound the potential energy by the standard energy conservation. This immediately leads to the averaged potential energy decay
\begin{align*}
\lim\limits_{T\rightarrow \infty}\frac{1}{T}\int_0^T\int_{|x|\leq t+R}|\phi|^{p+1} dxdt =0,\quad \forall R>0.
\end{align*}
Then by using the translated scaling vector field $(t+R)\pa_t+ x\pa_x$ as multiplier, we can improve the averaged time decay of the potential energy to be uniform time decay, hence concluding Theorem \ref{thm:potential:decay}. We remark here that since the initial energy on the region $\{|x|>R\}$ goes to zero as $R\rightarrow\infty$, in view of finite speed of propagation, the above averaged potential energy decay holds true after taking the limit $R\rightarrow \infty$. By using the standard energy conservation and Gagliardo-Nirenberg inequality, this gives an alternative proof for the averaged decay of Lindblad and Tao.

For the polynomial decay estimates of Theorem \ref{thm:polynomial:decay}, we need to use weighted vector fields as multipliers. The main difficulty lies in the necessity that the method should work for all range of $p$, including the limiting case when $p=1$, where the nonlinear wave equation degenerates to linear Klein-Gordon equation. Inspired by the vector field method in \cite{Yang:mMKG} effective both for the wave equation and Klein-Gordon equation in the exterior region $\{t\leq |x|+1\}$, we show that the Lorentz rotation vector field $x\pa_t+t\pa_x$ is sufficient to derive the necessary weighted energy estimates through light lines in the exterior region. The inverse polynomial decay of the solution in the exterior region then follows by using a type of weighted Gagliardo-Nirenberg inequality (see Lemma \ref{lem:GN1} for details).

The situation inside the light cone $\{|x|\leq t+1\}$ is more involved. We instead use the following class of new vector fields
\begin{align*}
\b^{-1} (1+t-x)^{\b}(1+t+x)^{\a-1}(\pa_t-\pa_x)+\a^{-1} (1+t-x)^{\b-1}(1+t+x)^{\a}(\pa_t+\pa_x)
\end{align*}
as multipliers with positive constants $\a$, $\b$ verifying the assumptions in Theorem \ref{thm:polynomial:decay}. Note that $1<\a+\b<2$. The weights of these vector fields have order between $0$ and $1$. The use of such vector fields with low weights as multipliers also appeared in \cite{yang:NLW:ptdecay:3D:smallp} for the uniform bound of solution to the three dimensional nonlinear wave equation.

\textbf{Acknowledgments.} S. Yang is partially supported by NSFC-11701017.

\section{Preliminary and energy identities}
Recall the energy momentum tensor associated to solution $\phi$ of \eqref{eq:NLW:semi:1d}
\begin{align*}
\mathbb{T}_{\mu\nu}=\pa_\mu \phi\cdot \pa_\nu \phi-\f12 m_{\mu\nu} (\pa^\ga\phi\pa_\ga \phi +\frac{2}{p+1}|\phi|^{p+1})
\end{align*}
with $\pa_1=\pa_x$, $\pa_0=\pa_t$ and $\mu,\nu=0, 1$. Here $m_{\mu\nu}$ is the flat Minkowski metric on $\mathbb{R}^{1+1}$. In particular, we can compute that
\begin{align*}
\mathbb{T}_{01}=\mathbb{T}_{10}=\partial_x\phi\partial_t\phi, \quad
\mathbb{T}_{11}=\frac{|\partial_x\phi|^2}{2}+\frac{|\partial_t\phi|^2}{2}-\frac{|\phi|^{p+1}}{p+1},\quad
\mathbb{T}_{00}=\frac{|\partial_x\phi|^2}{2}+\frac{|\partial_t\phi|^2}{2}+\frac{|\phi|^{p+1}}{p+1}.
\end{align*}
 Define the null coordinates $$u=t+1-x, \quad v=t+1+x $$
  as well as the associated null frame
  $$ \L=\partial_x+\partial_t, \quad \Lu=\partial_t-\partial_x. $$
    In particular we have $$\L u=\Lu v=0, \quad \L v=\Lu u=2.$$ Moreover we can write that
\begin{align}
\label{eq:T11}
  &\mathbb{T}_{00}-\mathbb{T}_{11}=\frac{2|\phi|^{p+1}}{p+1},\quad \mathbb{T}_{00}-\mathbb{T}_{01}=\frac{|\Lu\phi|^2}{2}+\frac{|\phi|^{p+1}}{p+1},\quad  \mathbb{T}_{00}+\mathbb{T}_{01}=\frac{|\L\phi|^2}{2}+\frac{|\phi|^{p+1}}{p+1}.
\end{align}
For solution $\phi$ of the nonlinear wave equation \eqref{eq:NLW:semi:1d}, we observe the conservation laws
 \begin{align}
 \label{eq:T00}
  &\partial_t\mathbb{T}_{00}=\partial_x\mathbb{T}_{01},\quad \partial_t\mathbb{T}_{01}=\partial_x\mathbb{T}_{11}.
\end{align}
Applying the first energy conservation to the region bounded by $\{t=t\}$ and the initial line $\{t=0\}$, we derive the
the classical energy conservation
\begin{equation}
\label{eq:E0}
 \int_{\R}\mathbb{T}_{00}(t,x)dx=\frac{1}{2}\int_{\mathbb{R}}|\pa_t\phi(t, x)|^2+|\pa_x\phi(t, x)|^2+\frac{2}{p+1}|\phi(t, x)|^{p+1}dx=\f12 \mathcal{E}_{0}[\phi],\quad \forall t\in \mathbb{R}.
\end{equation}
We will also use the energy conservation adapted to the region bounded by $\{t=0\}$, $\{t=T\}$ and the out going null line $\{x=t+R\}$
\begin{align}
\label{eq:EC:out}
  &\int_{0}^{T}\left(\frac{|\L\phi|^2}{2}+\frac{|\phi|^{p+1}}{p+1}\right)\Big|_{x=t+R} dt=\int_{t+R}^{+\infty} \mathbb{T}_{00}(t,x)dx\Big|_{t=T}^{t=0}
  \leq \frac{1}{2}\mathcal{E}_0[\phi]
\end{align}
as well as the associated energy conservation applied to the incoming null line
\begin{align}
\label{eq:EC:in}
  &\int_{0}^{T}\left(\frac{|\Lb\phi|^2}{2}+\frac{|\phi|^{p+1}}{p+1}\right)\Big|_{-x=t+R} dt= \int^{-t-R}_{-\infty} \mathbb{T}_{00}(t,x)dx\Big|_{t=T}^{t=0}
  \leq \frac{1}{2}\mathcal{E}_0[\phi].
\end{align}
For convenience, we may also carry out the computations under the null frame. First of all, we can rewrite the equation as follows
$$ \Box=-\partial_t^2+\partial_x^2=-\L\Lu,\quad \Box\phi=-\L\Lu \phi= |\phi|^{p-1}\phi. $$
 We therefore can obtain the identities
 \begin{align}
 \label{eq:L}
  &\L(\Lu\phi)^2=2(\L\Lu \phi)(\Lu \phi)=-2|\phi|^{p-1}\phi\cdot\Lu \phi=-\frac{2}{p+1} \Lu (|\phi|^{p+1}) ,\\
  \label{eq:Lu}
  &\Lu(\L\phi)^2=2(\L\Lu \phi)(\L \phi)=-2|\phi|^{p-1}\phi\cdot\L \phi=-\frac{2}{p+1} \L(|\phi|^{p+1} ).
\end{align}
The first identity can be viewed as using the vector field $\Lu$ as multiplier while the second identity is equivalent to take the vector field $L$ as multiplier. We finally define a quantity
\begin{align}
 \label{defQ}
  &Q=(-\partial_t^2+\partial_x^2)\phi^2=-2|\partial_t\phi|^2+2|\partial_x\phi|^2+2|\phi|^{p+1}=-2(\L\phi)(\Lu\phi)+2|\phi|^{p+1},
\end{align}
which will be frequently used in the sequel.

The proof for Theorem \ref{thm:polynomial:decay} relies on weighted energy estimates through null lines obtained by using weighted vector fields as
multipliers.
For the pointwise decay estimates, we will rely on the following type of  
  Gagliardo-Nirenberg  inequality.
\begin{Lem}
\label{lem:GN1}
Let $a_1, a_2, \mu_1, \mu_2$ be constants such that
\[
a_1\geq 0, \quad a_2\geq 1,\quad 0\leq \mu_1\leq 1,\quad \mu_2\geq -\mu_1.
\]
Then for a function $g(t)\in H_{loc}^1([a_1,+\infty))$, 
it holds that
\begin{align*}
(t+a_2)^{\mu_1+\mu_2}|g(t)|^{p+3}
&\leq C   \int_{a_1}^{\infty} (s+a_2)^{\mu_2}|g(s)|^{p+1}ds  \cdot \int_{a_1}^{\infty}(s+a_2)^{\mu_1}|g'(s)|^{2}ds
\end{align*}
for all $t\geq a_1$ with some constant $C$ independent of $t$, $a_1$ and $ a_2$. Here we assume that each integral in the right hand side of the above inequality is finite.
\end{Lem}
\begin{proof}
Define the non-negative function
$$h(t)=(t+a_2)^{\frac{\mu_1+\mu_2}{2}}|g(t)|^{\frac{p+3}{2}}.$$
 Then we have
\begin{align*}
h'(t)&=\frac{\mu_1+\mu_2}{2}(t+a_2)^{\frac{\mu_1+\mu_2}{2}-1}|g(t)|^{\frac{p+3}{2}}+\frac{p+3}{2}(t+a_2)^{\frac{\mu_1+\mu_2}{2}}|g(t)|^{\frac{p-1}{2}}
g(t)g'(t)
\\ &\geq-\frac{p+3}{2}(t+a_2)^{\frac{\mu_1+\mu_2}{2}}|g(t)|^{\frac{p+1}{2}}|g'(t)|.
\end{align*}
Here we note that $ \mu_1+\mu_2\geq 0.$ Now using H\"older inequality we have
\begin{equation}
\label{eq:h'L1}
\begin{split}
\|\max(-h'(t),0)\|_{L^1([a_1,+\infty))} &\leq C\|(t+a_2)^{\frac{\mu_1+\mu_2}{2}}|g(t)|^{\frac{p+1}{2}}|g'(t)|\|_{L^1([a_1,+\infty))}\\
&\leq C\|(t+a_2)^{\mu_2}|g|^{p+1}\|_{L^1([a_1,+\infty))}^{1/2}\|(t+a_2)^{\mu_1}|g'|^{2}\|_{L^1([a_1,+\infty))}^{1/2}.
\end{split}
\end{equation}
For $t_2\geq t_1\geq a_1$ we have
\begin{align*}
&h(t_1)-h(t_2)=\int_{t_1}^{t_2}-h'(t)dt\leq \int_{t_1}^{t_2}\max(-h'(t),0)dt\leq\|\max(-h'(t),0)\|_{L^1([a_1,+\infty))}.
\end{align*}
Therefore
\begin{align}
\label{eq:ht1}
&h(t_1)\leq\inf_{t\geq t_1}h(t)+\|\max(-h'(t),0)\|_{L^1([a_1,+\infty))},\quad \forall\ t_1\geq a_1.
\end{align}
Now for fixed $t_1\geq a_1,$ we claim that $ \inf_{t\geq t_1}h(t)=0$. Otherwise if $\inf_{t\geq t_1}h(t)=c_0 $ for some positive constant $c_0$, then we derive that
\begin{align*}
&h(t)=(t+a_2)^{\frac{\mu_1+\mu_2}{2}}|g(t)|^{\frac{p+3}{2}}\geq c_0,\quad |g(t)|\geq (t+a_2)^{-\frac{\mu_1+\mu_2}{p+3}}c_0^{\frac{2}{p+3}},\\& (t+a_2)^{\mu_2}|g|^{p+1}\geq (t+a_2)^{\mu_2-\frac{p+1}{p+3}(\mu_1+\mu_2)}c_0^{\frac{2(p+1)}{p+3}},\quad \forall\ t\geq t_1.
\end{align*}
As $\mu_1+\mu_2\geq 0$ and $ \mu_1\leq 1, $ we then conclude that
\begin{align*}
&{\mu_2-\frac{p+1}{p+3}(\mu_1+\mu_2)}\geq \mu_2-(\mu_1+\mu_2)\geq -\mu_1\geq -1.
\end{align*}
Now from the assumption $a_2\geq 1,\ t_1\geq a_1\geq 0$, we show that
\begin{align*}
\|(t+a_2)^{\mu_2}|g|^{p+1}\|_{L^1([a_1,+\infty))} & \geq\|(t+a_2)^{\mu_2}|g|^{p+1}\|_{L^1([t_1,+\infty))}
\\
& \geq \|(t+a_2)^{\mu_2-\frac{p+1}{p+3}(\mu_1+\mu_2)}\|_{L^1([t_1,+\infty))}c_0^{\frac{2(p+1)}{p+3}}\\
& \geq\|(t+a_2)^{-1}\|_{L^1([t_1,+\infty))}c_0^{\frac{2(p+1)}{p+3}}=+\infty,
\end{align*}
which contradicts the assumption that the first integral in the right hand side of the inequality of the Lemma is finite.
 Therefore we must have $ \inf_{t\geq t_1}h(t)=0$, which together with estimates \eqref{eq:ht1} and \eqref{eq:h'L1} leads to
\begin{align*}
h(t_1) &\leq\|\max(-h'(t),0)\|_{L^1([a_1,+\infty))}\\
&\leq C\|(t+a_2)^{\mu_2}|g|^{p+1}\|_{L^1([a_1,+\infty))}^{1/2}\|(t+a_2)^{\mu_1}|g'|^{2}\|_{L^1([a_1,+\infty))}^{1/2}.
\end{align*}
The Lemma then holds by definition of $h$ and squaring the above inequality.
\end{proof}
\begin{Remark}
\label{rem1}
For the special case when $a_1=\mu_1=\mu_2=0$, we derive from Lemma \ref{lem:GN1} that
\begin{align*}
|g(t)|^{\frac{p+3}{2}}
&\leq C\||g|^{p+1}\|_{L^1([0,+\infty))}^{1/2}\||g'|^{2}\|_{L^1([0,+\infty))}^{1/2},\ \forall\ t\geq 0.
\end{align*}
By symmetry we also have
\begin{align*}
|g(t)|^{\frac{p+3}{2}}
&\leq C\||g|^{p+1}\|_{L^1((-\infty,0])}^{1/2}\||g'|^{2}\|_{L^1((-\infty,0])}^{1/2},\ \forall\ t\leq 0.
\end{align*}
Now we obtain the classical Gagliardo-Nirenberg inequality
\begin{align*}
\|g\|_{L^{\infty}(\R)}^{({p+3})/{2}}
&\leq C\|g\|_{L^{p+1}(\R)}^{({p+1})/{2}}\|g'\|_{L^2(\R)}
\end{align*}
with constant $C$ depending only on $p$.
\end{Remark}
For finite energy solution of \eqref{eq:NLW:semi:1d}, the above classical Gagliardo-Nirenberg inequality together with the energy conservation \eqref{eq:E0} implies that
\begin{align*}
\|\phi(t, x)\|_{L_x^{\infty}(\R)}^{ p+3}
&\leq C \int_{\R} |\phi(t, x)|^{p+1} dx\cdot \int_{\R} |\pa_x\phi(t, x)|^2dx \leq\frac{p+1}{2}C \mathcal{E}_{0}[\phi]^2.
\end{align*}
In particular the finite energy solution verifies the following uniform bound
\begin{align}
\label{eq:unform:bd:phi}
|\phi|\leq C_p \mathcal{E}_0[\phi]^{\frac{2}{p+3}}.
\end{align}
This uniform bound is crucial during the proof. Moreover we see that the solution decays to $0$ once we have the potential energy decay for finite energy solutions.

Since the wave equation is time reversible, without loss of generality we only prove estimates in the future $t\geq 0$. Hence in the sequel we always assume that $t\geq 0$ unless it is specified.

\section{Proof of Theorem \ref{thm:potential:decay}}
In this section, the implicit constant in $A\les B$ depends only on the power $p$ and the standard energy $\mathcal{E}_{0}[\phi]$.

The key new observation to prove Theorem \ref{thm:potential:decay} is the time decay of the potential energy. By using the scaling vector field as multiplier applied to a truncated forward light cone, it follows from a type of weaker averaged decay estimate of the potential energy, which indeed is a consequence of the following the following integrated local potential energy decay estimate.
\begin{Prop}
\label{prop:Mora}
Let $\phi$ be a finite energy solution to \eqref{eq:NLW:semi:1d}. Then the potential energy verifies the following decay estimate
\begin{align*}
  &\int_0^{+\infty}\int_{-t-1}^{t+1}\frac{((t+1)^2-x^2)|\phi(t,x)|^{p+1}}{(t+1)^3}dxdt\leq C < \infty
\end{align*}
for some constant $C$ depending only on $p$ and the energy $\mathcal{E}_0[\phi]$.
\end{Prop}
\begin{Remark}
Such integrated local energy estimate dates back to Morawetz \cite{mora1}, \cite{mora2}. The original method there can easily lead to a stronger estimate than that of the above Proposition in higher dimenions ($d\geq 3$).
Similar estimates are available for solutions of defocusing nonlinear Klein-Gordon equations in lower dimensions, see \cite{Nakanishi99:scattering:NKG:12D}.
\end{Remark}
\begin{proof}
Let $t_*=t+1$. Recall the null frame  $ \L=\partial_x+\partial_t$, $\Lu=\partial_t-\partial_x$ and the associated null coordinates $u=t_*-x$, $v=t_*+x$. In particular we have
$$\L (u/t_*)=-\L (v/t_*)=-u/t_*^2,\quad \Lu (v/t_*)=-\Lu(u/t_*)=-v/t_*^2 .$$
Then in view of the equation \eqref{eq:L}, we have
\begin{align*}
  \L(u^{2}t_*^{-2}(\Lu\phi)^2)+\frac{2}{p+1} \Lu(u^{2} |\phi|^{p+1} t_*^{-2})&=\L(u^{2}t_*^{-2})(\Lu\phi)^2+\frac{2}{p+1}\Lu(u^{2}t_*^{-2}) |\phi|^{p+1}\\
  &=-2u^{2}t_*^{-3}(\Lu\phi)^2+\frac{4}{p+1} uv t_*^{-3} |\phi|^{p+1}.
\end{align*}
Similarly by \eqref{eq:Lu}, we also have
\begin{align*}
&\Lu(v^{2}t_*^{-2}(\L\phi)^2)+\frac{2}{p+1}\L (t_*^{-2}v^{2}|\phi|^{p+1})
=-2v^{2}t_*^{-3}(\L\phi)^2+\frac{4}{p+1} uvt_*^{-3} |\phi|^{p+1}.
\end{align*}
Combining these two identities, we end up with
\begin{align*}
&\L\left(\frac{u^{2}(\Lu\phi)^2}{t_*^2}+\frac{2v^{2}|\phi|^{p+1}}{t_*^2(p+1)}\right)
+\Lu\left(\frac{v^{2}(\L\phi)^2}{t_*^2}+\frac{2u^{2}|\phi|^{p+1}}{t_*^2(p+1)}\right)\\
&=-2u^{2}t_*^{-3}(\Lu\phi)^2-2v^{2}t_*^{-3}(\L\phi)^2+\frac{8}{p+1} uvt_*^{-3} |\phi|^{p+1}.
\end{align*}
Now we introduce
\begin{align*}
&P_3:=\frac{u^{2}(\Lu\phi)^2}{t_*^2}+\frac{2v^{2}|\phi|^{p+1}}{t_*^2(p+1)},\quad  P_4:=\frac{v^{2}(\L\phi)^2}{t_*^2}+\frac{2u^{2}|\phi|^{p+1}}{t_*^2(p+1)}.
\end{align*}
And recall $Q=\Box \phi^2$ as in \eqref{defQ}. Then the previous inequality leads to
\begin{align*}
&\L P_3
+\Lu P_4-2uvt_*^{-3}Q\\
&=-2u^{2}t_*^{-3}(\Lu\phi)^2-2v^{2}t_*^{-3}(\L\phi)^2+4uvt_*^{-3}(\L\phi)(\Lu\phi)-\frac{4(p-1)}{p+1} uvt_*^{-3} |\phi|^{p+1}
\\
&=-2t_*^{-3}|u\Lu\phi-v\L\phi|^2-\frac{4(p-1)}{p+1} uvt_*^{-3}  |\phi|^{p+1} \\
&\leq -\frac{4(p-1)}{p+1} uvt_*^{-3}  |\phi|^{p+1}.
\end{align*}
Integrating this inequality on the domain
\[
\Sigma_{1}^{T}:=\{(t, x):\,0\leq t\leq T,\quad |x|\leq t+1\},\quad T>0
\]
and using Stokes formula, we obtain that
\begin{align}
\nonumber
  &\frac{p-1}{p+1}\int_{\Sigma_{1}^{T}}4uvt_*^{-3}|\phi|^{p+1}dxdt-\int_{\Sigma_{1}^{T}}2uvt_*^{-3}Qdxdt\\
  \notag
  &\leq-\int_{\Sigma_1^T}(\L P_3
+\Lu P_4)dxdt\\
\label{eq:QP3P4}
&=-\int_{\Sigma_{1}^{T}}(\partial_t(P_3+P_4)+\partial_x(P_3-P_4)
)dxdt\\
\notag
&=\int_{\partial\Sigma_{1}^{T}}((P_3+P_4)dx-(P_3-P_4)dt).
\end{align}
Since $Q=\Box \phi^2$, by Stokes formula again, we have
\begin{align}
\label{eq:Suv}
  &\int_{\Sigma_{1}^{T}}uvt_*^{-3}Qdxdt=\int_{\Sigma_{1}^{T}}uvt_*^{-3}\Box \phi^2dxdt=\int_{\Sigma_{1}^{T}}\phi^2\Box(uvt_*^{-3}) dxdt+\int_{\partial\Sigma_{1}^{T}}\omega_1,
  \end{align}
  where
  \begin{align*}
&\omega_1:=(uvt_*^{-3}\partial_x (\phi^2)-\phi^2\partial_x(uvt_*^{-3}))dt+(uvt_*^{-3}\partial_t (\phi^2)-\phi^2\partial_t(uvt_*^{-3}))dx.
\end{align*}
Now for $|x|\leq t+1=t_*$, we compute that
\begin{align*}
&\Box(uvt_*^{-3})=\Box(t_*^{-1}-x^2t_*^{-3})=-4t_*^{-3}+12x^2t_*^{-5}\leq 8t_*^{-3}.
\end{align*}
Using the uniform bound \eqref{eq:unform:bd:phi} of the solution, we conclude that
\begin{align}
\label{eq:Sab}
  &\int_{\Sigma_{1}^{T}}\phi^2\Box(uvt_*^{-3}) dxdt \les  \int_{1}^{T}\int_{-T-1}^{T+1}(t+1)^{-3}dxdt \les 1.
\end{align}
Now we need to compute the boundary integrals. The boundary $\pa\Sigma_{1}^{T}$ consists of the null segments
$$\Gamma_1''=\{x=t+1,\ 0\leq t\leq T\},\quad \Gamma_2''=\{-x=t+1,\ 0\leq t\leq T\}$$
 and the constant $t$-slice
  $$\Gamma_3''=\{t=0,\ |x|\leq 1\},\quad \Gamma_4''=\{t=T,\ |x|\leq T+1\}.$$
   On the null segment $\Gamma_1''\subset\{x=t+1\}$, we have $u=(\partial_x+\partial_t)u=0$  and
\begin{align*}
&\omega_1=-\phi^2(\partial_x+\partial_t)(uvt_*^{-3})dt=0.
\end{align*}
Similarly on  the null segment $\Gamma_2''$, we have $ \omega_1=0.$
On the constant $t$-slice $\Gamma_3''$, $\Gamma_4''$, we have
\begin{align*}
&\omega_1=(uvt_*^{-3}\partial_t (\phi^2)-\phi^2\partial_t(uvt_*^{-3}))dx.
\end{align*}
Now for $|x|\leq t_*=t+1$, we in particular have that
\begin{align*}
|uvt_*^{-3}|=|(t_*^2-x^2)t_*^{-3}|\leq t_*^{-1},\quad |\partial_t(uvt_*^{-3})|=|t_*^{-2}-3x^2t_*^{-4}|\leq 2t_*^{-2}.
\end{align*}
Then in view of the energy conservation \eqref{eq:E0} and the uniform bound of the solution \eqref{eq:unform:bd:phi},
we then can bound that
\begin{align*}
&\int_{-t-1}^{t+1}|uvt_*^{-3}\partial_t (\phi^2)-\phi^2\partial_t(uvt_*^{-3})|dx \\
&\les \int_{-t-1}^{t+1}t_*^{-1}|\partial_t \phi||\phi|+t_*^{-2}|\phi|^2 dx\\
& \les  1+t_*^{-1}  \int_{-t-1}^{t+1} |\pa_t\phi| dx\\
&\les 1+ t_*^{-\frac{1}{2}}  \left(\int_{-t-1}^{t+1}|\pa_t\phi|^2 dx \right)^{\frac{1}{2}}\\
&\les 1,\quad \forall t\geq 0.
\end{align*}
In particular the integration of $ \omega_1$ on $\Gamma_3''$, $\Gamma_4'' $ is uniformly bounded by a constant independent of $T$. Combining this with estimates \eqref{eq:Suv}, \eqref{eq:Sab} and \eqref{eq:QP3P4}, we have shown that
\begin{align}
  \label{eq:ILE:000}
  \frac{p-1}{p+1}\int_{\Sigma_{1}^{T}}4uvt_*^{-3}|\phi|^{p+1}dxdt\les 1+\int_{\partial\Sigma_{1}^{T}}(P_3+P_4)dx-(P_3-P_4)dt.
\end{align}
Now for the boundary integral on the right hand side, on the null segment $\Gamma_1''\subset\{x=t+1\}$, we have
\begin{align*}
&(P_3+P_4)dx-(P_3-P_4)dt=2P_4dt,
\end{align*}
while on the null segment $\Gamma_2''\subset\{-x=t+1\}$, the integrand becomes
\begin{align*}
&(P_3+P_4)dx-(P_3-P_4)dt=-2P_3dt.
\end{align*}
On the constant $t$-slice $\Gamma_3''$, $\Gamma_4''$ on which $dt=0$,
it is clear that
\begin{align*}
&(P_3+P_4)dx-(P_3-P_4)dt=(P_3+P_4)dx.
\end{align*}
Summing up and considering the orientation we can show that\begin{align*}
  &\int_{\partial\Sigma_{1}^{T}}((P_3+P_4)dx-(P_3-P_4)dt)\\
  &=\int_{-t-1}^{t+1}(P_3+P_4)dx\Big|_{t=T}^{t=0}
  +\int_{0}^{T}2P_4(t,t+1)dt+\int_{0}^{T}2P_3(t,-t-1)dt.
\end{align*}
On the outgoing null segment $\{x=t+1\}$, by definition, we have
\begin{align*}
P_4(t, t+1)= 4|L(\phi)|^2.
\end{align*}
By using the standard energy estimate \eqref{eq:EC:out} adapted to the outgoing null line $\{x=t+1\}$, we conclude that
\begin{align*}
\int_0^T P_4(t, t+1)dt \leq 4 \int_{\{x=t+1\}} |L(\phi)|^2+\frac{2|\phi|^{p+1}}{p+1}dt \leq 8 \mathcal{E}_0[\phi].
\end{align*}
Similarly on the incoming null lines $\{x=-t-1\}$, we also have that
\begin{align*}
&\int_{0}^{T} P_3(t,-t-1)dt\leq 8\mathcal{E}_{0}[\phi].
\end{align*}
Next on the constant t-slice such that $|x|\leq t+1$, we show that
\begin{align*}
P_3+P_4&=t_*^{-2} u^2 |\Lb(\phi)|^2+t_*^{-2} v^2|L(\phi)|^2+\frac{2(u^2+v^2)|\phi|^{p+1}}{t_*^2(p+1)}\\
&\leq 4(|L(\phi)|^2+|\Lb(\phi)|^2)+\frac{8|\phi|^{p+1}}{p+1}.
\end{align*}
Hence by energy conservation \eqref{eq:E0}, we can bound that
\begin{align*}
\left|\int_{-t-1}^{t+1}(P_3+P_4)dx\Big|_{t=T}^{t=0}\right|\leq 16\mathcal{E}_0[\phi].
\end{align*}
Then in view of inequality \eqref{eq:ILE:000}, we conclude that
\begin{align*}
  \frac{p-1}{p+1}\int_{\Sigma_{1}^{T}}4uvt_*^{-3}|\phi|^{p+1}dxdt\les 1.
\end{align*}
Recall that $t_*=t+1$, $uv=t_*^2-x^2$. By the definition of $\Sigma_{1}^{T} $, the above estimate implies that
\begin{align*}
  &\int_0^{T}\int_{-t-1}^{t+1}\frac{((t+1)^2-x^2)|\phi(t,x)|^{p+1}}{(t+1)^3}dxdt=\int_{\Sigma_{1}^{T}}uvt_*^{-3}|\phi|^{p+1}dxdt\les 1.
\end{align*}
The proposition then follows by letting $ T\to+\infty$ and the convention that the implicit constant relies only on the power $p$ and the energy $\mathcal{E}_0[\phi]$.
\end{proof}

We now use the above integrated local potential energy decay to show the averaged decay for the potential energy.
\begin{Prop}
\label{prop:phi}
Let $\phi$ be a finite energy solution to the defocusing semilinear wave equation \eqref{eq:NLW:semi:1d}. Then for fixed $R>0,$ the potential energy verifies the averaged decay estimate
\begin{align*}
  &\lim_{T\to+\infty}\frac{1}{T}\int_0^{T}\int_{|x|\leq t+R}|\phi(t,x)|^{p+1}dxdt=0.
\end{align*}
\end{Prop}
\begin{proof}
For fixed $ \epsilon\in(0,1),$ by Proposition \ref{prop:Mora} there exists $T_0>0$ such that
\begin{align*}
  &\int_{T_0}^{+\infty}\int_{-t-1}^{t+1}\frac{((t+1)^2-x^2)|\phi(t,x)|^{p+1}}{(t+1)^3}dxdt\leq \epsilon^2.
\end{align*}
Restricting this integral on a smaller region, we conclude that
\begin{align}
\label{eq:T0}
  &\int_{T_0}^{+\infty}\int_{|x|\leq(t+1)(1-\epsilon)}\frac{|\phi(t,x)|^{p+1}}{(t+1)}dxdt\leq \epsilon
\end{align}
by noting that
 \begin{align*}
  &\frac{(t+1)^2-x^2}{(t+1)^3}\geq \frac{(t+1)-|x|}{(t+1)^2}\geq \frac{\epsilon}{(t+1)},\quad\text{for}\quad |x|\leq (t+1)(1-\epsilon).
\end{align*}
Now for $T>T_0$, we control the integrated potential energy by three different parts
\begin{align*}
  &\int_0^{T}\int_{|x|\leq t+R}|\phi(t,x)|^{p+1}dxdt\leq \mathrm{I}+\mathrm{II}+\mathrm{III},
\end{align*}
where
\begin{align*}
 \mathrm{I}=\int_0^{T}\int_{t-T_0-T\epsilon}^{ t+R} & |\phi(t,x)|^{p+1}dxdt,\quad \mathrm{II}=\int_0^{T}\int_{-t-R}^{T_0+T\epsilon-t}|\phi(t,x)|^{p+1}dxdt,
  \\
  \mathrm{III}&=\int_{T_0}^{T}\int_{|x|\leq(t+1)(1-\epsilon)}{|\phi(t,x)|^{p+1}}dxdt.
\end{align*}
To see the above claim, for the case when $t\leq T_0$, we have the inclusion
\begin{align*}
\{|x|\leq t+R\}\subset \{ t-T_0-T\epsilon\leq x\leq t+R\}\cup \{ -t-R\leq x\leq T_0+T\epsilon-t\}
\end{align*}
as $T_0+T\epsilon-t\geq 0$. For the other case when $T_0\leq t\leq T$, by symmetry, it suffices to check the relation
\begin{align*}
(t+1)(1-\epsilon)\geq t-T_0-T\epsilon,
\end{align*}
which holds as $0<\epsilon<1$ and $t\leq T$.

Now by the above estimate \eqref{eq:T0}, we can bound that
\begin{align*}
  &\mathrm{III}\leq(T+1)\int_{T_0}^{T}\int_{|x|\leq(t+1)(1-\epsilon)}\frac{|\phi(t,x)|^{p+1}}{(t+1)}dxdt\leq \epsilon(T+1).
\end{align*}
To control $\mathrm{I}$, we make use of the standard energy estimate \eqref{eq:EC:out} through the out going null lines to deduce that
\begin{align*}
  &\mathrm{I}=\int_{-T_0-T\epsilon}^{ R}\int_0^T|\phi(t,t+a)|^{p+1}dtda\leq\int_{-T_0-T\epsilon}^{ R}(p+1)\mathcal{E}_{0}[\phi]da=(R+T_0+T\epsilon)(p+1)\mathcal{E}_{0}[\phi].
\end{align*}
Similarly or by symmetry we also have that
\begin{align*}
  &\mathrm{II}\leq(R+T_0+T\epsilon)(p+1)\mathcal{E}_{0}[\phi].
\end{align*}
Summing up, we have shown that
\begin{align*}
  &\int_0^{T}\int_{|x|\leq t+R}|\phi(t,x)|^{p+1}dxdt\leq 2(R+T_0+T\epsilon)(p+1)\mathcal{E}_{0}[\phi]+\epsilon(T+1),
\end{align*}
which implies\begin{align*}
  &\limsup_{T\to+\infty}\frac{1}{T}\int_0^{T}\int_{|x|\leq t+R}|\phi(t,x)|^{p+1}dxdt\leq2\epsilon(p+1)\mathcal{E}_{0}[\phi]+\epsilon,\quad \forall\ \epsilon\in(0,1).
\end{align*}
This completes the proof by letting $ \epsilon\to 0.$
\end{proof}

We are now ready to prove the main Theorem \ref{thm:potential:decay}, for which we need to improve the above averaged decay of the potential energy to decay in the pointwise sense.  This is based on weighted energy estimate obtained by using the scaling vector field as multiplier.

 For fixed $R>0,$ in view of the energy identities \eqref{eq:T00}, \eqref{eq:T11}, we have
 \begin{align*}
  &\partial_t((t+R)\mathbb{T}_{00}+x\mathbb{T}_{01})-\partial_x((t+R)\mathbb{T}_{01}+x\mathbb{T}_{11})=\mathbb{T}_{00}-\mathbb{T}_{11}
  =\frac{2|\phi|^{p+1}}{p+1}.
\end{align*}
Integrating this identity on the domain
\[
\Sigma_{R}^{T}:=\{(t, x):\,0\leq t\leq T,\quad |x|\leq t+R\},\quad T>0,\ R>0
\]
 and using Stokes formula, we obtain that
 \begin{align*}
  &\int_{\partial\Sigma_{R}^{T}}[((t+R)\mathbb{T}_{00}+x\mathbb{T}_{01})dx+((t+R)\mathbb{T}_{01}+x\mathbb{T}_{11})dt]
  =-\int_{\Sigma_{R}^{T}}\frac{2|\phi|^{p+1}}{p+1}dxdt.
\end{align*}
The boundary $\pa\Sigma_{R}^{T}$ consists of the null segments
$$\Gamma_1''=\{x=t+R,\ 0\leq t\leq T\},\quad \Gamma_2''=\{-x=t+R,\ 0\leq t\leq T\}$$
and the constant $t$-slice
$$\Gamma_3''=\{t=0,\ |x|\leq R\},\quad \Gamma_4''=\{t=T,\ |x|\leq T+R\}.$$
 For the null segment $\Gamma_1''\subset\{x=t+R\}$, we can compute that (using \eqref{eq:T11})
\begin{align*}
&((t+R)\mathbb{T}_{00}+x\mathbb{T}_{01})dx+((t+R)\mathbb{T}_{01}+x\mathbb{T}_{11})dt=x(\mathbb{T}_{00}+2\mathbb{T}_{01}+\mathbb{T}_{11})dt
=(t+R)|\L\phi|^2dt.
\end{align*}
Similarly for the null segment $\Gamma_2''\subset\{-x=t+R\}$, we instead have
\begin{align*}
&((t+R)\mathbb{T}_{00}+x\mathbb{T}_{01})dx+((t+R)\mathbb{T}_{01}+x\mathbb{T}_{11})dt=x(\mathbb{T}_{00}-2\mathbb{T}_{01}+\mathbb{T}_{11})dt
=-(t+R)|\Lu\phi|^2dt.
\end{align*}
Then on the constant $t$-slice $\Gamma_3'',\ \Gamma_4''$, we have
\begin{align*}
&((t+R)\mathbb{T}_{00}+x\mathbb{T}_{01})dx+((t+R)\mathbb{T}_{01}+x\mathbb{T}_{11})dt=((t+R)\mathbb{T}_{00}+x\mathbb{T}_{01})dx.
\end{align*}
Therefore the above energy identity reads as
\begin{align*}
  &\int_{-t-R}^{t+R}((t+R)\mathbb{T}_{00}(t,x)+x\mathbb{T}_{01}(t,x))dx\Big|_{t=0}^{t=T}
  \\=&\int_{0}^{T}(t+R)|\L\phi(t,t+R)|^2dt+\int_{0}^{T}(t+R)|\Lu\phi(t,-t-R)|^2dt+\int_{\Sigma_{R}^{T}}\frac{2|\phi(t,x)|^{p+1}}{p+1}dxdt.
\end{align*}
For the first term on the right hand side, we make use of the standard energy conservation \eqref{eq:EC:out}
 adapted to the out going null lines
 to bound that
\begin{align*}
  \int_{0}^{T}(t+R)|\L\phi(t,t+R)|^2dt & \leq (T+R)\int_{0}^{T}|\L\phi(t,t+R)|^2dt\\
  & \leq 2(T+R)\int_{t+R}^{+\infty} \mathbb{T}_{00}(t,x)dx\Big|_{t=T}^{t=0}.
\end{align*}
Similarly or by symmetry, we also have
\begin{align*}
  &\int_{0}^{T}(t+R)|\Lu\phi(t,-t-R)|^2dt\leq 2(T+R)\int_{-\infty}^{-t-R} \mathbb{T}_{00}(t,x)dx\Big|_{t=T}^{t=0}.
\end{align*}
On the other hand, for $|x|\leq t+R$, we can estimate that
\begin{align*}
  &(t+R)\mathbb{T}_{00}+x\mathbb{T}_{01}\geq (t+R)(\mathbb{T}_{00}-|\mathbb{T}_{01}|)\geq(t+R)\frac{|\phi|^{p+1}}{p+1},\\&(t+R)\mathbb{T}_{00}+x\mathbb{T}_{01}\leq 2(t+R)\mathbb{T}_{00}.
\end{align*}
This implies that
\begin{align*}
  &\int_{-t-R}^{t+R}((t+R)\mathbb{T}_{00}(t,x)+x\mathbb{T}_{01}(t,x))dx\Big|_{t=0}^{t=T}\\
  &\geq (T+R)\int_{|x|\leq T+ R}\frac{|\phi(T,x)|^{p+1}}{p+1}dx-2R\int_{|x|\leq R}\mathbb{T}_{00}(0,x)dx.
\end{align*}
Combining the above estimates, we conclude that
\begin{align*}
  &(T+R)\int_{|x|\leq T+ R}\frac{|\phi(T,x)|^{p+1}}{p+1}dx-2R\int_{|x|\leq R}\mathbb{T}_{00}(0,x)dx\\
  &\leq 2(T+R)\int_{|x|\geq t+R}\mathbb{T}_{00}(t,x)dx\Big|_{t=T}^{t=0}+\int_{\Sigma_{R}^{T}}\frac{2|\phi(t,x)|^{p+1}}{p+1}dxdt.
\end{align*}
In particular we derive that
\begin{align*}
  &(T+R)\int_{|x|\leq T+ R}\frac{|\phi(T,x)|^{p+1}}{p+1}dx+2(T+R)\int_{|x|\geq T+R}\mathbb{T}_{00}(T,x)dx\\
  &\leq 2R\int_{|x|\leq R}\mathbb{T}_{00}(0,x)dx+2(T+R)\int_{|x|\geq R}\mathbb{T}_{00}(0,x)dx+\int_{\Sigma_{R}^{T}}\frac{2|\phi(t,x)|^{p+1}}{p+1}dxdt\\
  &= 2R\int_{\R}\mathbb{T}_{00}(0,x)dx+2T\int_{|x|\geq R}\mathbb{T}_{00}(0,x)dx+\int_{\Sigma_{R}^{T}}\frac{2|\phi(t,x)|^{p+1}}{p+1}dxdt.
\end{align*}
Now since
\[
\mathbb{T}_{00}\geq\frac{1}{p+1} |\phi|^{p+1},
\]
the previous estimate leads to
\begin{align*}
  &(T+R)\int_{\R}\frac{|\phi(T,x)|^{p+1}}{p+1}dx\leq R\mathcal{E}_{0}[\phi]+2T\int_{|x|\geq R}\mathbb{T}_{00}(0,x)dx+\int_{\Sigma_{R}^{T}}\frac{2|\phi(t,x)|^{p+1}}{p+1}dxdt.
\end{align*}
Dividing by $T+R$ and taking limit in terms of $T$, we end up with
\begin{align*}
  &\limsup_{T\to+\infty}\int_{\R}{|\phi(T,x)|^{p+1}}dx\leq 2(p+1)\int_{|x|\geq R}\mathbb{T}_{00}(0,x)dx+\limsup_{T\to+\infty}\frac{2}{T}\int_{\Sigma_{R}^{T}}{|\phi(t,x)|^{p+1}}dxdt.
\end{align*}
Now by using Proposition \ref{prop:phi} and the definition of $\Sigma_{R}^{T} $, we have
\begin{align*}
  &\lim_{T\to+\infty}\frac{1}{T}\int_{\Sigma_{R}^{T}}{|\phi(t,x)|^{p+1}}dxdt=\lim_{T\to+\infty}\frac{1}{T} \int_0^T \int_{|x|\leq t+R} |\phi(t,x)|^{p+1}dxdt=0.
\end{align*}
Therefore
\begin{align*}
  &\limsup_{T\to+\infty}\int_{\R}{|\phi(T,x)|^{p+1}}dx\leq 2(p+1)\int_{|x|\geq R}\mathbb{T}_{00}(0,x)dx,\quad \forall\ R>0.
\end{align*}
Since the initial energy is finite
\[
\int_{|x|\geq R}\mathbb{T}_{00}(0, x)dx\leq \int_{\mathbb{R}} \mathbb{T}_{00}(0, x)dx=\frac{1}{2}\mathcal{E}_0[\phi]<\infty,
\]
by letting $R\to+\infty$, we conclude that
\begin{align*}
  &\lim_{T\to+\infty}\int_{\R}{|\phi(T,x)|^{p+1}}dx=0,\quad \text{i.e.}\quad \lim_{t\to+\infty}\|\phi(t)\|_{L_x^{p+1}(\R)}^{p+1}=0.
\end{align*}
Once we have this potential energy decay, the pointwise decay estimate follows from Gagliardo-Nirenberg inequality (see Remark \ref{rem1}) together with the energy conservation
\begin{align*}
  &\|\phi(t)\|_{L_x^{\infty}(\R)}^{p+3}\les \|\partial_x\phi(t)\|_{L_x^{2}(\R)}^{2}\|\phi(t)\|_{L_x^{p+1}(\R)}^{p+1}\les \mathcal{E}_{0}[\phi]\|\phi(t)\|_{L_x^{p+1}(\R)}^{p+1}.
\end{align*}
This in particular implies that
\begin{align*}
  &\lim_{t\to+\infty}\|\phi(t)\|_{L_x^{\infty}(\R)}= \lim_{t\to+\infty}\|\phi(t)\|_{L_x^{p+1}(\R)}=0.
\end{align*}
By symmetry, this limit is also true in the past as $t\to-\infty$. This completes the proof for Theorem \ref{thm:potential:decay}.

\section{Proof for Theorem \ref{thm:polynomial:decay}}

In this section,  we make a convention that $A\les B$ means there is a constant $C$ depending only  on $p$, $\a$, $ \b$ and the weighted energy $\mathcal{E}_{1}[\phi]$ such that $A\leq CB$. Here the constants $\a$, $\b$ verify the assumption in Theorem \ref{thm:polynomial:decay}.

As we have pointed out in the introduction, the improved inverse polynomial decay of the solution follows from weighted energy estimates together with a type of weighted Gagliardo-Nirenberg inequality of Lemma \ref{lem:GN1}. In the exterior region $|x|\leq t$, it suffices to make use of the Lorentz vector field $t\pa_x+x\pa_t$ as multiplier to deduce weighted energy estimates through out going null lines. This, however, works only in the exterior region. In the interior region $\{|x|\leq t\}$, we need to use new weighted vector fields as multipliers to derive the necessary weighted energy estimates through out going and incoming null lines.


\subsection{Decay estimates in the exterior region}
In view of the energy identity \eqref{eq:T00}, we in particular have that
\begin{align*}
  &\partial_t((x+1)\mathbb{T}_{00}+t\mathbb{T}_{01})=\partial_x((x+1)\mathbb{T}_{01}+t\mathbb{T}_{11}).
\end{align*}
Integrate this equality on the domain
\[
\mathcal{D}_{a}^{b}:=\{(t, x):\,t\geq0,\quad t-x\leq -a,\quad t+x\leq b\},\quad a<b.
\]
Using Stokes formula, we obtain that
\begin{align*}
  &\int_{\partial\mathcal{D}_{a}^{b}}[((x+1)\mathbb{T}_{00}+t\mathbb{T}_{01})dx+((x+1)\mathbb{T}_{01}+t\mathbb{T}_{11})dt]=0.
\end{align*}
The boundary $\pa\mathcal{D}_{a}^{b}$ consists of the null segments
$$\Gamma_1=\{t-x=-a,\ 0\leq t\leq \frac{b-a}{2}\},\quad \Gamma_2=\{t+x=b,\ 0\leq t\leq \frac{b-a}{2}\}.$$
 and the constant $t$-slice $\Gamma_3=\{t=0,\ a\leq x\leq b\}$.

 Now we compute the boundary integrals. For the outgoing null segment $\Gamma_1\subset\{t-x=-a\}$, we can compute that
\begin{align*}
&((x+1)\mathbb{T}_{00}+t\mathbb{T}_{01})dx+((x+1)\mathbb{T}_{01}+t\mathbb{T}_{11})dt\\
= &((x+1)\mathbb{T}_{00}+(t+x+1)\mathbb{T}_{01}+t\mathbb{T}_{11})dt\\
=& \f12 \left((t+x+1) |\L\phi|^2+\frac{2(x+1-t)}{p+1}  |\phi|^{p+1} \right)dt\\
=& \f12 \left((2t+a+1) |\L\phi|^2+ \frac{2(a+1)}{p+1}  |\phi|^{p+1} \right)dt.
\end{align*}
Similarly for the incoming null segment $\Gamma_2\subset\{t+x=b\}$, we show that
\begin{align*}
&((x+1)\mathbb{T}_{00}+t\mathbb{T}_{01})dx+((x+1)\mathbb{T}_{01}+t\mathbb{T}_{11})dt\\
=&(-(x+1)\mathbb{T}_{00}+(x+1-t)\mathbb{T}_{01}+t\mathbb{T}_{11})dt\\
=&-\f12 \left((x+1-t) |\Lu\phi|^2+\frac{2(x+1+t)}{p+1} |\phi|^{p+1}\right)dt\\
=&-\f12 \left ((b+1-2t) |\Lu\phi|^2 +\frac{2(b+1)}{p+1}  |\phi|^{p+1}\right)dt.
\end{align*}
Finally on the constant $t$-slice $\Gamma_3\subset\{t=0\}$, we have
\begin{align*}
&((x+1)\mathbb{T}_{00}+t\mathbb{T}_{01})dx+((x+1)\mathbb{T}_{01}+t\mathbb{T}_{11})dt=(x+1)\mathbb{T}_{00}dx.
\end{align*}
Summing up and considering the orientation, we obtain that
\begin{align*}
  &\int_0^{\frac{b-a}{2}}\left(\frac{2(a+1)|\phi|^{p+1}}{p+1}+(2t+a+1) |\L\phi|^2 \right)|_{x=t+a}dt
  \\&+\int_0^{\frac{b-a}{2}} \left(\frac{2(b+1)|\phi|^{p+1}}{p+1}+(b+1-2t) |\Lu\phi|^2\right)|_{x=b-t}dt\\=&
  2\int_{a}^b(x+1)\mathbb{T}_{00}(0,x)dx\leq  2\int_{\R}(|x|+1)\mathbb{T}_{00}(0,x)dx=\mathcal{E}_{1}[\phi].
\end{align*}
If $a\geq 0$ then each term on the left hand side is nonnegative. By letting $b\to+\infty$, we conclude that
\begin{align}
\label{eq:a>0}
  &\int_0^{+\infty}(\frac{2(a+1)}{p+1}|\phi|^{p+1}+(2t+a+1)|\L\phi|^2)|_{x=t+a}dt\leq \mathcal{E}_{1}[\phi],\quad \forall a\geq 0.
\end{align}
By symmetry (changing $x$ to $-x$), we also have
\begin{align}
\label{eq:a<0}
  &\int_0^{+\infty}(\frac{2(a+1)}{p+1}|\phi|^{p+1}+(2t+a+1)|\Lu\phi|^2)|_{x=-a-t}dt\leq \mathcal{E}_{1}[\phi],\quad \forall a\geq 0.
\end{align}
Alternatively one can also subtract the previous energy identity with four times of the standard energy (obtained by using $\pa_t$ as multiplier) and then let $b\leq 0$ and $a\rightarrow -\infty$.

Now for fixed $a\geq 0$,  define $g(t)=\phi(t,t+a)$, $a_1=0$, $a_2=a+1$, $\mu_1=1$,  $\mu_2=0$. We in particular have that $g'(t)=\L\phi(t,t+a)$. Then in view of the above estimate \eqref{eq:a>0}, we derive that
\begin{align*}
  &(a+1)\||g|^{p+1}\|_{L^1([0,+\infty))}+\|(t+a_2)|g'|^{2}\|_{L^1([0,+\infty))}\leq \frac{p+1}{2}\mathcal{E}_{1}[\phi].
\end{align*}
By Lemma \ref{lem:GN1}, we conclude that
\begin{align*}
  &(t+a_2)^{\frac{1}{2}}|g(t)|^{\frac{p+3}{2}}
\leq C\||g|^{p+1}\|_{L^1([0,+\infty))}^{1/2}\|(t+a_2)|g'|^{2}\|_{L^1([0,+\infty))}^{1/2}\les (a+1)^{-\frac{1}{2}}.
\end{align*}
This in particular leads to the pointwise decay estimate
\begin{align*}
&|\phi(t,t+a)|
\les  (t+a+1)^{-\frac{1}{p+3}}(a+1)^{-\frac{1}{p+3}},\quad \forall\ t\geq 0,\ a\geq 0,
\end{align*}
which implies the first inequality of Theorem \ref{thm:polynomial:decay} for the case $x\geq t\geq 0.$ By symmetry (or changing variables $x\to -x$) the first inequality holds on the whole exterior region $|x|\geq |t|\geq 0.$

\subsection{Decay estimates in the interior region}
The decay estimates in the interior region are much more difficult to obtain. We introduce new vector fields as multipliers.

Now for $u=t+1-x>0$, $v=t+1+x>0$ and constants $\a$, $\b$ verifying the assumptions in Theorem \ref{thm:polynomial:decay}, in view of the identity \eqref{eq:L}, we can compute that
\begin{align*}
  \L(u^{\b}v^{\a-1}(\Lu\phi)^2)+ \frac{2}{p+1}\Lu (u^{\b}v^{\a-1}|\phi|^{p+1}) &=\L(u^{\b}v^{\a-1})(\Lu\phi)^2+ \frac{2}{p+1}\Lu(u^{\b}v^{\a-1})  |\phi|^{p+1}
  \\
  &=2(\a-1)u^{\b}v^{\a-2}(\Lu\phi)^2+ \frac{4}{p+1} \b u^{\b-1}v^{\a-1}  |\phi|^{p+1}.
\end{align*}
Similarly (from \eqref{eq:Lu}) we also have
\begin{align*}
&\Lu(u^{\b-1}v^{\a }(\L\phi)^2)+ \frac{2}{p+1}\L  ( u^{\b-1}v^{\a}|\phi|^{p+1} )
=2(\b-1)u^{\b-2}v^{\a}(\L\phi)^2+\frac{4}{p+1} \a u^{\b-1}v^{\a-1}  |\phi|^{p+1}.
\end{align*}
Introduce the following quantities
\begin{align*}
&P_1=\frac{u^{\b}v^{\a-1}(\Lu\phi)^2}{\b}+\frac{2u^{\b-1}v^{\a}|\phi|^{p+1}}{\a(p+1)},\quad
P_2=\frac{u^{\b-1}v^{\a}(\L\phi)^2}{\a}+\frac{2u^{\b}v^{\a-1}|\phi|^{p+1}}{\b(p+1)}.
\end{align*}
Here recall that $\a$, $\b$ are positive constants.
Combining the above two identities, we end up with
\begin{align*}
&\L P_1+\Lu P_2-\frac{4}{p+1} u^{\b-1}v^{\a-1} Q \\
&= \frac{ 2(\b-1)}{\a}u^{\b-2}v^{\a}(\L\phi)^2+\frac{2(\a-1)}{\b}u^{\b}v^{\a-2}(\Lu\phi)^2+\frac{8}{p+1} u^{\b-1}v^{\a-1} (\L\phi)(\Lu\phi)
\\
&=-2u^{\b-2}v^{\a-2}\big|\sqrt{{\a}^{-1}(1-\b)}\cdot v\L\phi-\sqrt{\b^{-1}(1-\a) }\cdot u\Lu\phi\big|^2\\
& \leq  0.{}
\end{align*}
Here we used the relation
\begin{align*}
{\a}^{-1}(1-\a){\b}^{-1}(1-\b) =\frac{4}{(p+1)^2},\quad 0<\a, \b<1
\end{align*}
assumed in the theorem and $Q$ is defined in \eqref{defQ}.
 Integrate the above inequality on the domain
\[
\mathcal{R}_{a}^{b}:=\{(t, x):\,0\leq t-x\leq a,\quad 0\leq t+x\leq b\},\quad a>0,\ b>0.
\]
By using Stokes formula, we obtain that
\begin{align}
\nonumber
  \frac{4}{p+1}\int_{\mathcal{R}_{a}^{b}}u^{\b-1}v^{\a-1} Q dxdt&\geq\int_{\mathcal{R}_{a}^{b}}(\L P_1
+\Lu P_2)dxdt=\int_{\mathcal{R}_{a}^{b}}(\partial_t(P_1+P_2)+\partial_x(P_1-P_2)
)dxdt\\
\label{eq:QP1P2}
&=\int_{\partial\mathcal{R}_{a}^{b}}((P_1-P_2)dt-(P_1+P_2)dx).
\end{align}
Recall that  $Q=\Box \phi^2$. By using Stokes formula again, we can compute that
\begin{align}
\notag
  \int_{\mathcal{R}_{a}^{b}}u^{\b-1}v^{\a-1}Qdxdt&=\int_{\mathcal{R}_{a}^{b}}u^{\b-1}v^{\a-1}\Box \phi^2dxdt\\
  \label{eq:Rabuv}
  &=\int_{\mathcal{R}_{a}^{b}}\phi^2\Box(u^{\b-1}v^{\a-1}) dxdt+\int_{\partial\mathcal{R}_{a}^{b}}\omega,
  \end{align}
  in which
  \begin{align*}
 \omega=(u^{\b-1}v^{\a-1}\partial_x (\phi^2)-\phi^2\partial_x(u^{\b-1}v^{\a-1}))dt+(u^{\b-1}v^{\a-1}\partial_t (\phi^2)-\phi^2\partial_t(u^{\b-1}v^{\a-1}))dx.
\end{align*}
Now note that
$$\Box(u^{\b-1}v^{\a-1})=-\L\Lu(u^{\b-1}v^{\a-1})=-4(\b-1)(\a-1)u^{\b-2}v^{\a-2}<0$$
by the assumption that $0<\a, \b<1$. In particular this implies that
\begin{align}
\label{eq:Rab}
  &\int_{\mathcal{R}_{a}^{b}}\phi^2\Box(u^{\b-1}v^{\a-1}) dxdt\leq 0.
\end{align}
We now compute that boundary terms.
The boundary $\pa\mathcal{R}_{a}^{b}$ of the domain consists of the four null segments
\begin{align*}
\Gamma_1' &=\{t-x=0,\ 0\leq t\leq \frac{b}{2}\},\quad \Gamma_2'=\{t+x=b,\ \frac{b}{2}\leq t\leq \frac{b+a}{2}\},\\
 \Gamma_3'&=\{t-x=a,\ \frac{a}{2}\leq t\leq \frac{b+a}{2}\},\quad  \Gamma_4'=\{t+x=0,\ 0\leq t\leq \frac{a}{2}\}.
\end{align*}
 For the null segment $\Gamma_3'\subset\{t=x+a\}$, we can compute that
\begin{align*}
\omega=&(u^{\b-1}v^{\a-1}(\partial_x+\partial_t) (\phi^2)-\phi^2(\partial_x+\partial_t)(u^{\b-1}v^{\a-1}))dt=(g_1' g_2-g_1g_2')dt.
\end{align*}
Here we denote
\begin{align*}
&g_1(t) =\phi^2(t,t-a),\quad g_2(t):=(u^{\b-1}v^{\a-1})(t, t-a) =(a+1)^{\b-1}(2t-a+1)^{\a-1}.
\end{align*}
Then its integration on $\Gamma_1' $ becomes
\begin{align*}
&\int_{\Gamma_3'}\omega=\int_{\frac{a}{2}}^{\frac{a+b}{2}}(g_1'g_2-g_1g_2')dt=(g_1g_2)|_{t=\frac{a}{2}}^{t=\frac{a+b}{2}}-
2\int_{\frac{a}{2}}^{\frac{a+b}{2}}g_1g_2'dt.
\end{align*}
Since $0<\a, \b<1$, $a>0$ and $t\geq \frac{a}{2}$, we have
$$0<g_2\leq 1,\quad  g_2'<0.$$
Then by using the uniform bound \eqref{eq:unform:bd:phi} for finite energy solution,
we can bound that
\begin{align*}
&\left|\int_{\Gamma_3'}\omega\right|\les 1+\int_{\frac{a}{2}}^{\frac{a+b}{2}}|g_2'|dt\les 1.
\end{align*}
By our notation, the implicit constant relies only on the weighted energy $\mathcal{E}_1[\phi]$, the power $p$ and the constants $\a$, $\b$. In particular, the above uniform bound is independent of $a$, $b$. By setting $a=0$, we also have the uniform bound for the integration of $\omega$ on the null segment $\Gamma_1'$. In a similar argument (or by symmetry $x\rightarrow -x$), the integration of $ \omega$ on other incoming null segments $\Gamma_2'$, $\Gamma_4' $ is also uniformly bounded by a constant independent of $a$ and $b$,
that is,
\begin{align*}
\left| \int_{\partial\mathcal{R}_{a}^{b}}\omega\right|\les 1.
\end{align*}
Combining this estimate with \eqref{eq:Rabuv}, \eqref{eq:Rab} and \eqref{eq:QP1P2}, we conclude that
\begin{align*}
  &\int_{\mathcal{R}_{a}^{b}}u^{\b-1}v^{\a-1}Qdxdt\les 1,\quad \int_{\partial\mathcal{R}_{a}^{b}}((P_1-P_2)dt-(P_1+P_2)dx)\les 1.
\end{align*}
Now on the outgoing null segments $\Gamma_1'$ and $\Gamma_3'$, we compute that
\begin{align*}
dx=dt,\quad (P_1-P_2)dt-(P_1+P_2)dx=-2P_2dt,
\end{align*}
while on the incoming null segments $\Gamma_2'$ and $ \Gamma_4'$, we have
\begin{align*}
dx=-dt,\quad (P_1-P_2)dt-(P_1+P_2)dx=2P_1dt.
\end{align*}
Summing up and considering the orientation we have shown that
\begin{align}
\notag
  &-\int_0^{\frac{b}{2}}2P_2(t,t)dt+\int_{\frac{b}{2}}^{\frac{b+a}{2}}2P_1(t,b-t)dt
  +\int_{\frac{a}{2}}^{\frac{b+a}{2}}2P_2(t,t-a)dt-\int_0^{\frac{a}{2}}2P_1(t,-t)dt\\
  \label{eq:w:ei}
  &=\int_{\partial\mathcal{R}_{a}^{b}}((P_1-P_2)dt-(P_1+P_2)dx)\\
  \notag
  &\les 1.
\end{align}
Note that for $t\geq 0$, we have $(u)(t, t)=1$ and $(v)(t, t)=2t+1\geq 1$. Hence we have
\begin{align*}
  P_2(t,t)&=\Big(\frac{u^{\b-1}v^{\a}(\L\phi)^2}{\a}+\frac{2u^{\b}v^{\a-1}|\phi|^{p+1}}{\b(p+1)}\Big)\Big|_{x=t}\\
  &=\Big(\frac{v^{\a}(\L\phi)^2}{\a}+\frac{2v^{\a-1}|\phi|^{p+1}}{\b(p+1)}\Big)\Big|_{x=t}\\
  &\les (2t+1)(\L\phi)^2(t, t)+|\phi|^{p+1}(t, t).
\end{align*}
In view of estimate \eqref{eq:a>0} with $a=0$, we derive that
\begin{align*}
  &\int_0^{\frac{b}{2}}2P_2(t,t)dt\les \int_0^{+\infty}({(2t+1)(\L\phi)^2}+{|\phi|^{p+1}})|_{x=t}dt\les 1.
\end{align*}
Similarly by using estimate \eqref{eq:a<0} with $a=0$ (or simply by symmetry again), we also have
\begin{align*}
  &\int_0^{\frac{a}{2}}2P_1(t,-t)dt\les 1.
\end{align*}
Then from the above inequality \eqref{eq:w:ei}, we conclude that
\begin{align*}
  &\int_{\frac{b}{2}}^{\frac{b+a}{2}}2P_1(t,b-t)dt
  +\int_{\frac{a}{2}}^{\frac{b+a}{2}}2P_2(t,t-a)dt\les 1,\quad \forall a, b>0.
\end{align*}
Since $P_1$ and $P_2$ are non-negative, by letting $b\to+\infty$, we derive that
\begin{align*}
  &\int_{\frac{a}{2}}^{+\infty}2P_2(t,t-a)dt\les 1.
\end{align*}
Note that for $t\geq a>a/2>0$ it holds that
\[
(u)(t, t-a)=a+1,\quad t+1\leq (v)(t, t-a)=2t-a+1\leq  2t+1.
\]
This implies that for $t\geq a$ we have the lower bound
\begin{align*}
  P_2(t,t-a)&=\Big(\frac{u^{\b-1}v^{\a}(\L\phi)^2}{\a}+\frac{2u^{\b}v^{\a-1}|\phi|^{p+1}}{\b(p+1)}\Big)\Big|_{x=t-a}
  \\ & \geq \a^{-1}(a+1)^{\b-1}(t+1)^{\a}(\L\phi)^2+2(p+1)^{-1}\b^{-1}(a+1)^{\b}(2t+1)^{\a-1}|\phi|^{p+1}.
\end{align*}
Here we note that $0<\a<1$.
By restricting the integral on the region where $t\geq a$, we derive from the previous estimate that
\begin{align*}
  &\int_{a}^{+\infty}((a+1)^{\b-1}(t+1)^{\a}(\L\phi)^2+(a+1)^{\b}(t+1)^{\a-1}|\phi|^{p+1})|_{x=t-a}dt\les 1.
\end{align*}
Now for fixed $a> 0$,  let
 $$g(t)=\phi(t,t-a), \quad a_1=a,\quad a_2=1,\quad \mu_1=\a \in[1/2,1),\quad \mu_2=\a-1. $$
In particular we have
$$ \mu_1+\mu_2=2\a-1\geq 0, \quad g'(t)=(\L\phi)(t,t-a).$$
Therefore the above weighted energy estimate through the outgoing null lines reads as
\begin{align*}
(a_1+1)^{\b-1} \int_{a_1}^{+\infty}(t+a_2)^{\mu_1}|g'(t)|^2dt+(a_1+1)^{\b}\int_{a_1}^{\infty} (t+a_2)^{\mu_2}|g|^{p+1} dt\les 1.
\end{align*}
Then using Lemma \ref{lem:GN1}, we conclude that
\begin{align*}
  (t+a_2)^{\mu_1+\mu_2} |g(t)|^{p+3}
\leq C  \int_{a_1}^{\infty} (s+a_2)^{\mu_2} |g|^{p+1} ds\cdot \int_{a_1}^{\infty}(s+a_2)^{\mu_1}|g'|^{2}ds \les (a_1+1)^{1-2\b}.
\end{align*}
By definition, this implies that
\begin{align*}
|\phi(t,t-a)|
\les (t+1)^{\frac{1-2\a}{p+3}}(a+1)^{\frac{1-2\b}{p+3}}, \quad \forall t\geq a> 0.
\end{align*}
Hence the pointwise decay estimate of Theorem \ref{thm:polynomial:decay} in the interior region holds for the case when $t\geq x\geq 0.$ By symmetry (changing $x$ to $-x$) again, the same estimate is valid for $-t\leq x\leq 0.$

Finally to show the uniform inverse polynomial time decay of the solution, in the exterior region where $t\leq |x|$, the decay estimate in the exterior region in the previous subsection implies that
\begin{align*}
  &|\phi(t,x)|\les (t+1)^{-\frac{1}{p+3}}
  \les (1+|t|)^{-\frac{p-1}{(p+1)^2+4}}.
\end{align*}
In the interior region $|x|\leq t$, take the constants $\a$, $\b$ as follows
\[
\b=\frac{1}{2},\quad \a^{-1}=1+\frac{4}{(p+1)^2}<2.
\]
This verifies the condition in Theorem \ref{thm:polynomial:decay}. Then the pointwise decay estimate in the interior region shows that
\begin{align*}
|\phi(t,x)|
\les (t+1)^{\frac{1-2\a}{p+3}}=(1+t)^{-\frac{p-1}{(p+1)^2+4}}.
\end{align*}
Combining this with the above estimate in the exterior region, we therefore have shown that on the whole future spacetime, it always holds that
\begin{align*}
|\phi(t,x)|
\les (t+1)^{\frac{1-2\a}{p+3}}=(1+t)^{-\frac{p-1}{(p+1)^2+4}}.
\end{align*}
This completes the proof for Theorem \ref{thm:polynomial:decay}.

\bibliography{shiwu}{}
\bibliographystyle{plain}


\end{document}